\documentclass[twoside,12pt]{amsart}
\usepackage{amssymb}
\usepackage{amsmath}
\usepackage{fancyhdr}
\usepackage{amsthm}
\usepackage{rotate}
\usepackage{graphicx}
\usepackage[T1]{fontenc}
\usepackage{afterpage}  %

\newtheorem{num}{Notation}[section]

\newtheorem{proposition}[num]{Proposition}

\newtheorem{lemma}[num]{Lemma}
\newtheorem{corollary}[num]{Corollary}

\newtheorem{conjecture}[num]{Conjecture}

\usepackage{lscape}
\usepackage[all]{xy}
\usepackage{pifont}
\usepackage{bbding}
\usepackage{url}

\def\SL{ \text{\rm SL} }
\def\PSL{ \text{\rm PSL} }
\def\Hom{ \text{\rm Hom} }

\begin{document}
\date{\today}
\title{A note on invariant transversals for normal subgroups}

\author{Gerhard Hiss}

\address{Lehrstuhl f\"ur Algebra und Zahlentheorie, RWTH Aachen University,
52056 Aach\-en, Germany}

\email{gerhard.hiss@math.rwth-aachen.de}

\subjclass[2020]{Primary: 20E45; Secondary: 20C25, 20N05}
\keywords{Transversals, normal subgroups, conjugacy classes, commutators}

\begin{abstract}
The existence of invariant transversals for a normal subgroup~$H$ in a group~$G$
is investigated. This yields counterexamples to a conjecture in case~$H$ is 
abelian and $G$ is finite. 
\end{abstract}

\maketitle

%

\section{Introduction}

Let~$G$ be a group and let $H, L \leq G$ be subgroups of~$G$. Let
$H\backslash G := \{ Hg \mid g \in G \}$ denote the set of right $H$-cosets 
in~$G$. A transversal for $H\backslash G$ is a set of representatives. A 
subset~$S$ of~$G$ is called $L$-invariant, if~$S$ is invariant under 
conjugation by~$L$ or, in other words, if~$S$ is a union of $L$-conjugacy 
classes of~$G$. If there exists an $L$-invariant transversal 
for~$H\backslash G$, we say that~$H$ admits an $L$-invariant transversal. A 
$G$-invariant transversal for $H\backslash G$ is, in particular, a loop 
transversal, i.e., a transversal for the cosets of $g^{-1}Hg\backslash G$ for 
all $g \in G$. Loop transversals carry the structure of a loop.

In~\cite{HO}, the authors put forward the following conjecture, where~$G'$
denotes the commutator subgroup of~$G$. 

\begin{conjecture}
\label{Conjecture1}
Let~$G$ be finite and let $H \leq G$ be abelian admitting a $G$-invariant 
transversal. Then $H \cap G' = \{ 1 \}$.
\end{conjecture}

This conjecture, for which the author of this note takes full responsibility, 
already appears in Ortjohann's Master thesis~\cite{LO}. What was the motivation 
for this conjecture? Firstly, this was the observation that if 
$H \cap G' = \{ 1 \}$ (in which case~$H$ is necessarily abelian), then any
transversal~$S$ of $HG'\backslash G$ yields the $G$-invariant transversal
$T = \{ gs \mid g \in G', s \in S \}$ for $H\backslash G$. Thus the truth of 
Conjecture~\ref{Conjecture1} would have simplified the classification of the 
instances $(G,H,T)$ 
with~$H$ abelian and~$T$ a $G$-invariant transversal for $H\backslash G$. 
Secondly, Artic in her PhD-thesis \cite{articdiss} determined the pairs 
$(G,H)$ with~$G$ finite, $[G\colon\!H] \leq 30$ and core-free~$H$, admitting a 
$G$-invariant transversal. All of Artic's examples satisfy 
Conjecture~\ref{Conjecture1}. In her Master thesis~\cite{LO}, Ortjohann verified 
the conjecture for all groups~$G$ of order at most~$40$. Finally, 
\cite[Theorem~$3.1$]{HO} reproduces a positive result of Kochend{\"o}rffer and 
Zappa, in case~$H$ is an abelian Hall subgroup of~$G$.

The main purpose of this note is to exhibit counterexamples to 
Conjecture~\ref{Conjecture1}. We came across these in our attempt to prove the
conjecture for abelian normal subgroups~$H$. 
There is a second reason 
for looking at normal subgroups. Suppose that~$T$ is a $G$-invariant transversal 
for $H\backslash G$ and that~$L$ is a subgroup of~$G$ with $H \leq L$. Then 
$T \cap L$ is an $L$-invariant transversal for $H\backslash L$. This reasoning,
which does not assume that~$H$ is abelian, applies in particular when~$L$ is the 
normalizer of~$H$ in~$G$.

In Section~\ref{Characterizations}, not assuming that the groups are 
finite, we give a straightforward criterion for the existence of invariant 
transversals for normal subgroups. In Section~\ref{FiniteGroups} we discuss 
finite groups, where the conjecture leads to a question on central extensions 
and corresponding projective representations. Fortunately, this question had 
been answered a long time ago. The answer yields counterexamples to 
Conjecture~\ref{Conjecture1}. Gunter Malle called the attention of the author
to the results in~\cite{commut} and, via his survey article~\cite{MalleOre},
to earlier results~\cite{blau} by Blau. The latter leads to two non-solvable
counterexamples.

\section{Characterizations}
\label{Characterizations}

In this section~$G$ is a group and $H \unlhd G$ is a normal subgroup of~$G$. 
Our notation is standard. In particular, $Z(G)$ and $G'$ denote, respectively, 
the center of~$G$ and the commutator subgroup of~$G$. For $g, x \in G$, we put 
$g^x := x^{-1}gx$ and $[g,x] := g^{-1}x^{-1}gx$. Also, $C_G( x )$ and $C_G( L )$
denote the centralizer in~$G$ of $x \in G$, respectively the 
subgroup~$L \leq G$. Since $H \unlhd G$, we have $H\backslash G = G/H$, and we 
write $\pi\colon G \rightarrow G/H$ for the canonical epimorphism.

\begin{lemma}
\label{CentralizerLem}
Let $T$ be an $H$-invariant transversal for $G/H$. 
Then $\langle T \rangle \leq C_G( H )$ and $G = H \langle T \rangle = HC_G(H)$.
\end{lemma}
\begin{proof}
Since $H \unlhd G$, we have $Hg = Hg^h$ for all $g \in G$ and $h \in H$. Since 
$T$ is $H$-invariant, this yields $t = t^h$ for all $t \in T$ and all $h \in H$,
and thus our first assertion.

From this and $G = \langle H, T \rangle$, we obtain 
$G = H \langle T \rangle = HC_G(H)$.
\end{proof}

\begin{corollary}
\label{CentralizerCor}
Let~$H$ be abelian. Then~$H$ admits an $H$-invariant transversal, if and only if 
$H \leq Z(G)$. 
\end{corollary}
\begin{proof}
Suppose first that~$H$ admits an $H$-invariant transversal. Since~$H$ is 
abelian, $H \leq C_G( H )$ and so $C_G( H ) = HC_G( H ) = G$ by
Lemma~\ref{CentralizerLem}. Thus $H \leq Z(G)$. 

Conversely, if $H \leq Z( G )$, every transversal for $G/H$ is $H$-invariant.
This concludes the proof.
\end{proof}

Let $T \subseteq G$. By definition,~$T$ is a transversal for~$G/H$, if and only
if~$\pi$ restricts to a bijection $\pi|_T\colon T \rightarrow G/H$. If~$T$ is a
$G$-invariant transversal for~$G/H$, then~$\pi|_T$ induces a bijection between 
the conjugacy classes of~$G$ contained in~$T$ and the conjugacy classes 
of~$G/H$. The existence of $G$-invariant transversals for normal subgroups is 
characterized by the following result.

\begin{proposition}
\label{CharacterizingGInvariance}
There is a $G$-invariant transversal for $G/H$, if and only if
\begin{equation}
\label{FactorizationOfG}
G = HC_G(H) 
\end{equation}
and
\begin{equation}
\label{CentralizerEquation}
C_{G/H}( Hg ) = HC_G( g )/H\quad\quad\text{for all\ } g \in G.
\end{equation}
\end{proposition}
\begin{proof}
Suppose first that~$T$ is a $G$-invariant transversal for~$G/H$. Then 
$G = HC_G(H)$ by Lemma~\ref{CentralizerLem}, so~(\ref{FactorizationOfG}) holds.

To prove~(\ref{CentralizerEquation}), let $g \in G$. It suffices to show that
$C_{G/H}( Hg ) \leq HC_G( g )/H$. For this purpose, let $x \in G$ with $[x,g] 
\in H$. Write $g = ht$ and $x = h't'$ with $h, h' \in H$ and $t, t' \in T$. 
Since $t, t' \in C_G( H )$, once more by Lemma~\ref{CentralizerLem}, we obtain
$[x,g] = [t',t][h',h]$, and so $[t',t] \in H$. It follows that 
$Ht = tH = t^{t'}H = Ht^{t'}$. Since~$T$ is $G$-invariant, we conclude that~$t'$
centralizes~$t$, and thus~$g$. Hence $x \in HC_G( g )$, as claimed.

Now suppose that~(\ref{FactorizationOfG}) and~(\ref{CentralizerEquation}) are
satisfied. For $g \in G$, let $C_g$ denote the $G$-conjugacy class of~$g$. 
Choose a 
transversal~$S$ for $G/H$ with $S \subseteq C_G(H)$. For $s \in S$ we have 
$H \leq C_G( s )$, and thus $\pi( C_G( s ) ) = C_{G/H}( \pi(s) )$ 
from~(\ref{CentralizerEquation}). It follows that~$\pi$ restricts to a bijection 
between~$C_s$ and the $G/H$-conjugacy class of~$\pi(s)$.

Let $\overline{\mathcal{C}}$ denote the set of conjugacy classes of $G/H$. For 
every $\bar{C} \in \overline{\mathcal{C}}$, choose $s({\bar{C}}) \in S$ with 
$\pi( s({\bar{C}}) ) \in \bar{C}$. 
Then 
$$
T := \bigcup_{\bar{C} \in \overline{\mathcal{C}}} C_{s({\bar{C}})}
$$
is $G$-invariant. Moreover,~$T$ is a transversal for $G/H$, since~$\pi$ maps~$T$
bijectively onto~$G/H$.
\end{proof}

The next results prepares for a reduction.

\begin{lemma}
\label{ReductionToHAbelian}
Suppose that $G = H C_G(H)$ and let $T \subseteq G$. Then~$T$ is a $G$-invariant 
transversal for~$G/H$ if and only if $T \subseteq C_G( H )$ and $T$ is a 
$C_G(H)$-invariant transversal for $C_G(H)/Z(H)$.
\end{lemma}
\begin{proof}
Suppose that~$T$ is a $G$-invariant transversal for $G/H$. Then 
$T \subseteq C_G(H)$ by Lemma~\ref{CentralizerLem}. Thus~$T$ is a 
$C_G( H )$-invariant transversal for $C_G(H)/(H\cap C_G(H))$. 

Now suppose that $T \subseteq C_G(H)$ is a $C_G( H )$-invariant transversal for 
$C_G(H)/Z(H)$. Then~$T$ is a transversal for~$G/H$. Moreover,~$T$ is 
$G$-invariant, since~$H$ centralizes~$C_G(H)$.
\end{proof}

Let~$T$ be a $G$-invariant transversal for $G/H$, and put $L := C_G( H )$. If 
$Z(H) = \{ 1 \}$ we obtain the uninteresting case $G = H \times L$. In general, 
$G = H \circ_Z L$ is a central product, where $Z = Z(H)$ is amalgamated with a 
central subgroup of~$L$, and $T \subseteq L$ is an $L$-invariant transversal 
for~$L/Z$. Thus the classification of the $G$-invariant transversals for~$G/H$ 
can be reduced to the case when~$H$ is abelian.

\begin{corollary}
\label{CharacterizingGInvarianceHAbelian}
Suppose that~$H$ is abelian. Then~$H$ admits a $G$-invariant transversal,
if and only if 
\begin{equation}
\label{FactorizationOfGHAbelian}
H \leq Z(G) 
\end{equation}
and 
\begin{equation}
\label{CentralizerEquationHAbelian}
C_{G/H}( Hg ) = C_G( g )/H\quad\quad \text{for all\ } g \in G.
\end{equation}
\end{corollary}
\begin{proof}
If~$H$ is abelian, $H \leq C_G( H )$, and so Condition~(\ref{FactorizationOfG})
is equivalent to~(\ref{FactorizationOfGHAbelian}). Moreover, if $H \leq Z(G)$, 
then $H \leq C_G( g )$ for all $g \in G$, so that 
Condition~(\ref{CentralizerEquation}) is equivalent 
to~(\ref{CentralizerEquationHAbelian}).
\end{proof}

Notice that if $H \leq Z(G)$, then Condition~(\ref{CentralizerEquationHAbelian}) 
is equivalent to the statement that no non-trivial commutator of~$G$ lies 
in~$H$. The latter condition holds for $G = \SL_2( \mathbb{R} )$ with
$H = Z( G )$. Since $\SL_2( \mathbb{R} )$ is perfect, this yields an infinite
``counterexample'' to Conjecture~\ref{Conjecture1}. In contrast, by a theorem of 
Ree~\cite{ree}, every element in a semisimple algebraic group over an 
algebraically closed field is a commutator. With a finite number of exceptions,
the same is true for every finite quasisimple group by the work of Liebeck,
O'Brien, Shalev and Tiep~\cite[Theorem~$1.1$]{commut}. (A finite group is 
quasisimple, if it is perfect and simple modulo its center.) As will be 
discussed at the end this note, two of these exceptions, already contained in 
the work of Blau~\cite{blau}, give rise to non-solvable counterexamples to 
Conjecture~\ref{Conjecture1}. 

\section{Finite groups and $G$-invariant transversals}
\label{FiniteGroups}

In this section~$G$ is a finite group and $H \unlhd G$. For a finite group~$L$,
let $k(L)$ denote the number of conjugacy classes of~$L$. In~\cite{gallagher}, 
Gallagher has shown that Condition~(\ref{CentralizerEquation}) is equivalent to
\begin{equation}
\label{GallagherRelation}
k(G) = k(G/H)\cdot k(H).
\end{equation}
The article of Gallagher also contains a character theoretic interpretation
of this relation.
A search with GAP~\cite{GAP4} using~(\ref{GallagherRelation}), discloses a 
group~$G$ of order~$64$ with a subgroup~$H$ of index~$2$ such that~$(G,H)$ 
satisfies~(\ref{CentralizerEquation}) but not~(\ref{FactorizationOfG}). If~$G$ 
is the dihedral group of order~$8$ and $H = Z(G)$, then $(G,H)$
satisfies~(\ref{FactorizationOfG}) but not~(\ref{CentralizerEquation}). Thus 
neither of these conditions implies the other one.

From now on, we write $\bar{G} := G/H$ and use the bar-convention to denote the
elements of~$\bar{G}$, i.e., $\bar{g} \in \bar{G}$ is the image of $g \in G$ 
under the canonical epimorphism $G \rightarrow \bar{G}$. The following 
reasoning, based on \cite{reynolds}, leads the way to counterexamples to 
Conjecture~\ref{Conjecture1}.

Suppose that $H \leq Z(G)$. Thus
$$1 \longrightarrow H \longrightarrow G \longrightarrow \bar{G} 
\longrightarrow 1$$
is a central extension of~$\bar{G}$ with kernel~$H$. 
Let 
$T = \{ t_{\bar{x}} \mid \bar{x} \in \bar{G} \}$
be a transversal for~$G/H$. From~$T$ we obtain a $2$-cocycle
$\gamma \in Z^2( \bar{G}, H )$ by the rule
$$t_{\bar{x}} t_{\bar{y}} = \gamma( \bar{x}, \bar{y} ) t_{\bar{x}\bar{y}}$$
for $\bar{x}, \bar{y} \in \bar{G}$. Let 
$\lambda \in \Hom( H, \mathbb{C}^{\times} )$. Then 
$\alpha := \lambda \circ \gamma$ is a $2$-cocycle in 
$Z^2( \bar{G}, \mathbb{C}^{\times} )$, and we write $[\alpha]$ 
for its image in $H^2( \bar{G}, \mathbb{C}^{\times} )$. Then~$[\alpha]$
is non-trivial, if and only if $H \cap Z( G' ) \not\leq \ker( \lambda )$; see 
\cite[Theorem~$(11.19)$]{Isaacs}. By definition, $\bar{x} \in \bar{G}$ is called
$\alpha$-regular, if
$\alpha( \bar{x}, \bar{y} ) = \alpha( \bar{y}, \bar{x} )$ for all
$\bar{y} \in C_{\bar{G}}( \bar{x} )$. 

Here is the link to invariant transversals. By the definition of~$\gamma$, the 
centralizer condition~(\ref{CentralizerEquationHAbelian}) is satisfied 
for~$(G,H)$, if and only if 
\begin{equation}
\label{gammaEquation}
\gamma( \bar{x}, \bar{y} ) = \gamma( \bar{y}, \bar{x} )\quad\text{for all\ }
\bar{x} \in \bar{G} \text{\ and all\ } \bar{y} \in C_{\bar{G}}( \bar{x} ).
\end{equation}
If~$\lambda$ is faithful,~(\ref{gammaEquation}) is equivalent to the condition 
that every $\bar{x} \in \bar{G}$ is $\alpha$-regular. If this holds, the number 
of irreducible complex characters of~$G$ which restrict to~$H$ as a multiple 
of~$\lambda$ equals the number of irreducible complex characters of~$\bar{G}$.
Can this happen if $[\alpha]$ is non-trivial? This is the question addressed by 
Reynolds in~\cite{reynolds}.

Let~$p$ be a prime. In \cite[\S$3$]{reynolds}, Reynolds constructs a metabelian 
group~$G$ of order~$p^8$ with a subgroup $H \leq Z(G) \cap G'$ of order~$p$, 
such that, for every $\lambda \in \Hom( H, \mathbb{C}^{\times} )$, every element 
of~$\bar{G}$ is $\alpha$-regular. It follows that the pair $(G,H)$ is a 
counterexample to Conjecture~\ref{Conjecture1}. The final section of Gallagher's 
article~\cite{gallagher} contains further counterexamples, which the author 
attributes to Dade. 

A computation with GAP shows that there are no counterexamples $(G,H)$ to 
Conjecture~\ref{Conjecture1} with $|G| < 128$ and $H \leq Z (G )$. On the other 
hand, we found~$52$ isomorphism classes of groups~$G$ of order~$2^7$, containing 
a subgroup $H \leq Z(G) \cap G'$ of order~$2$ such that
\begin{equation}
\label{GallagherSpecial}
k( G ) = k( G/H )\cdot 2.
\end{equation}
Thus~(\ref{GallagherRelation}), and hence~(\ref{CentralizerEquationHAbelian}) is 
satisfied for~$(G,H)$, so that these groups yield counterexamples as well. 
Let~$G$ be one of these groups. Then $Z( G ) \cong G/G'$ is elementary
abelian of order~$8$, $Z( G ) \leq G'$,~$G'$ is abelian
of exponent~$4$, and~$Z(G)$ contains a unique subgroup~$H$ of order~$2$ 
satisfying~(\ref{GallagherSpecial}).

We conclude this note with a going down result. 
\begin{lemma}
\label{MinimalCounterexample}
Let $H \leq Z(G)$. Suppose that~$p$ is a prime dividing $|H \cap G'|$. Let~$P$ 
be a Sylow $p$-subgroup of~$G$ and let $Q \leq H \cap G'$ be of order~$p$. Then 
$Q \leq P'$.

If~$H$ admits a $G$-invariant transversal, then~$Q$~admits a $P$-invariant 
transversal. 
\end{lemma}
\begin{proof}
The first assertion follows from a transfer result; see 
\cite[Satz IV.$2.2$]{huppert}.

If~$H$ admits a $G$-invariant transversal, there is a $G$-invariant transversal 
for~$G\backslash Q$, since $Q \leq H \leq Z(G)$. Intersecting this with~$P$, we
obtain a $P$-invariant transversal for $Q\backslash P$.
\end{proof}

In the following,~$C_n$ denotes a a cyclic group of order~$n$. Let~$G$ be 
quasisimple and let $\{ 1\} \lneq H \leq Z(G)$ be such that no non-trivial 
element of~$H$ is a commutator. Then $G \in \{ G_1, G_2 \}$, where~$G_i$ is a 
covering group of $\PSL_3(4)$ for $i = 1, 2$, $Z( G_1 ) \cong C_2 \times C_{12}$
and $Z( G_2 ) \cong C_2 \times C_{4}$; moreover, $|H| = 2$ 
(see~\cite[Theorem~$1$]{blau} or \cite[Theorem~$1.1$ and Table~$1$]{commut}). 
Up to isomorphism, the covering groups~$G_1$ and~$G_2$ of $\PSL_3(4)$ are 
determined by their centers. For~$G_1$, this follows from the fact that there
is an automorphism of the universal covering group of $\PSL_3(4)$ permuting the
three involutions in its center transitively; see~\cite[Table~$6.3.1$]{GLS}.
For the uniqueness of~$G_2$ see \cite[Theorem~$6.3.2$]{GLS}. 
Using GAP~\cite{GAP4}, one can show that $Z( G_i )$ 
contains a unique subgroup~$H_i$ of order~$2$ with the given property, 
$i = 1, 2$. Up to isomorphism, we thus obtain two non-solvable counterexamples 
$(G_1,H_1)$ and $(G_2,H_2)$ to Conjecture~\ref{Conjecture1}. Let~$S$ be a Sylow 
$2$-subgroup of~$G_1$. Then $|S| = 2^9$ and~$S$ is isomorphic to a Sylow 
$2$-subgroup of~$G_2$. By Lemma~\ref{MinimalCounterexample}, the pair $(S,H_1)$ 
also yields a counterexample.

\section*{Acknowledgements}
It is my pleasure to thank Xenia Flamm for her comments, which greatly
improved the exposition of this note. I also thank Gunter Malle for taking my
attention to References~\cite{commut} and~\cite{MalleOre}. Finally, I thank
Thomas Breuer for his help with the covering groups of $\PSL_3(4)$.

\end{document}